\documentclass[reqno]{amsart}
\usepackage{ amssymb }
\usepackage{geometry} 
 \numberwithin{equation}{section}

\def\R{\mathbb{R}}

\def\N{\mathbb{N}}

\usepackage{hyperref}

\usepackage{enumitem}

\newtheorem{theo}{Theorem}

  \newtheorem{defi}{Definition}[section]
  \newtheorem{prop}{Proposition}[section]

        \newtheorem{exa}{Example}[section]

\usepackage{hyperref}

\begin{document}
\title[Heat flow of $p$-harmonic maps]{Heat flow of $p$-harmonic maps from Complete Manifolds into Generalised Regular Balls}
\author{Zeina AL Dawoud}
\address{Laboratoire de Math\'ematiques, UMR 6205 CNRS
Universit\'e de Bretagne Occidentale
6 Avenue Le Gorgeu, 29238 Brest Cedex 3
France}
\email{Zeina.Aldawoud@univ-brest.fr}

\begin{abstract}
We study the heat flow of $p$-harmonic maps between complete Riemannian manifolds. We prove the global existence of the flow when the initial datum has values in a generalised regular ball.  In particular, if the target manifold has nonpositive sectional curvature, we obtain the global existence of the flow for any initial datum with finite $p$-energy. If, in addition, the target manifold is compact,  the flow  converges to a $p$-harmonic map. This gives  an extension of the results of Liao-Tam \cite{Liao Tam} concerning the harmonic heat flow ($p=2$) to the case $p\ge2$. We also derive a Liouville type theorem for $p$-harmonic maps between complete Riemannian manifolds.
\end{abstract}
\maketitle

\section{Introduction}
Let $(M^m,g)$ and $(N^n,h)$ be two Riemannian manifolds, with $M$ compact.  For $p >1$, the $p$-energy of a map $u\in C^1(M,N)$ is defined by
\begin{equation}\label{11}
    E_p(u)=\frac{1}{p}\int_M |du(x)|^p dx, 
\end{equation}
where $|du(x)|$  is the Hilbert-Schmidt norm of  $du(x)$, and $dx$ stands for the Riemannian volume element of $M$. 
\medskip

$p$-harmonic maps are  critical points of the functional  $(1.1)$, that is,  they are solutions of the Euler-Lagrange equation associated to $(1.1)$  
\begin{equation} \label{12}
   \tau_p(u)=0, 
\end{equation}
where $\tau_p(u)=\mathrm{Trace}_g(\nabla(|du|^{p-2}du))$ denotes the $p$-tension field of $u$.  More precisely, let $(x^1, \cdots, x^m)$ and $(y^1, \cdots, y^n )$ be local coordinates on $M$ and $N$ respectively, then   denoting $\partial_j= {\partial\over \partial x^j}$ and $u^{\alpha}=y^{\alpha}(u)$,  equation  \eqref{12} takes the form  of the following system 

\begin{equation} \label{13} 
  -\frac{1}{\sqrt{|g|}}\partial_i\left(|du|^{p-2} \sqrt{|g|} g^{ij}\partial_j u^k \right) = 
 |du|^{p-2}g^{ij}\Gamma_{\alpha\beta}^{k}(u) \partial_i u^{\alpha}\partial_j u^{\beta}, \ 1\leq k\leq  n, 
\end{equation}
 where $\Gamma^{l}_{\alpha\beta}$ are the  Christoffel symbols of $N$, and  the Einstein summation convention on repeated indices is used. \\

In order to consider more general  class  of solutions of  system \eqref{13}, we shall  write it in an equivalent form which is  independent of the choice of coordinates.  By Nash's embedding theorem, we can embed $N$ isometrically  in an Euclidean space $\mathbb{R}^{L}$ using an isommetric embedding $i : N \to \R^L$. If we note again $u=i\circ u$, the energy functional \eqref{11} becomes 
$$E_p(u)= {1\over p} \int_M |\nabla u(x)|^p dx,$$
where  $|\nabla u|^2 = g^{ij} \langle \partial_iu, \partial_ju\rangle$, and $\langle  .\ , \ . \rangle$ denotes the Euclidean innner product of $\mathbb{R}^{L}$.

\medskip

Equation \eqref{13} takes the form 

\begin{equation} \label{14}
    -\Delta_{p} u=|\nabla u|^{p-2}A(u)(\nabla u,\nabla u),
\end{equation}

\medskip

\noindent where $\Delta_{p}u$ $=$ div$(|\nabla u|^{p-2}\nabla u)$ is the $p$-Laplacian, that is, 

$$\Delta_{p}u=\frac{1}{\sqrt{|g|}}\partial_i\left(\sqrt{|g|}|\nabla u|^{p-2}g^{ij}\partial_j u\right), $$

\noindent and $A$ is the second fundamental form of $N$ in $\mathbb{R}^{L}$  with the notation

$$A(u)(\nabla u,\nabla u)= g^{ij}A(u)(\partial_iu, \partial_ju).  $$

 \medskip

  In this case, $p$-harmonic maps are defined to be solutions of equation \eqref{14}.  We note that equation \eqref{14} makes sense even if $M$ is not compact.  \\

\par We define the Sobolev space $W^{1,p}(M,N)=\{u\in W^{1,p}(M,\mathbb{R}^{L}); \ u(x)\in N \ a.e\}$.  We say that a map $u\in W^{1,p}(M,N) \cap L^{\infty}(M, \R^L)$  is a weakly $p$-harmonic map if it is a weak solution of \eqref{14}. The study of the regularity of weakly $p$-harmonic maps is a delicate question due to the fact that \eqref{14} is a degenerate quasilinear elliptic system.  In general, the optimal regularity for weak solutions to systems involving the $p$-Laplacian   is   $C^{1+\beta}$ as shown by \cite{Tolksdorf}, and $C^{\infty}$ out off the vanishing set of $\nabla u$.  Solutions which are   $p$-energy-minimzing  are in  $C^{1+\beta}(M\setminus S,N)$ $(0<\beta<1)$ where $S$ is a singular set of Hausdorff dimension at most $m-[p]-1$ (see \cite{Hardt} \cite{Luckhauss}). 

\medskip

 In this paper we are interested in  heat flow of $p$-harmonic maps which is  the gradient flow associated with the $p$-energy functional. Namely, 
 
\begin{equation} \label{15}
    \left\{
                \begin{array}{ll}
                  \partial_t u-\Delta_{p} u =|\nabla u|^{p-2}A(u)(\nabla u,\nabla u),\\
                  \\
                 u(x,0)=u_{0}(x), \\
            
                \end{array}
              \right. 
 \end{equation}
 \\
 
\noindent where $u_0: M\rightarrow N$ is the initial datum of the flow. \\

Throughout this paper,   what we mean by a solution  $u$ of \eqref{15}  on  $M \times [0, T)$ is a map $u \in C_{loc}^{1+ \beta, \hspace{0.5mm}\beta/p}(M\times [0, T), N)$ (for some $0 < \beta < 1$)  which  is a  weak solution (in the distributional sense) of \eqref{15}. Indeed, as stated above,  the optimal regularity one could expect for $p$-harmonic type equations is $C^{1+ \beta}$ (and $C^{1+ \beta, \hspace{0.5mm}\beta/p}$ for parabolic equations).  \\

\par When $p=2$, Eells and Sampson \cite{Eells Sampson} were the first who studied the heat flow problem of harmonic maps. They proved that if $M$ and $N$ are compact Riemannian manifolds  and $N$ has nonpositive sectional curvature,  then \eqref{15} admits a global solution which converges at infinity to a harmonic map.  Li and Tam \cite{Li Tam} considered the case when both $M$ and $N$ are complete noncompact Riemannian manifolds. They proved the existence of a global solution when the Ricci curvature of $M$ is bounded from below,  $N$ has  nonpositive sectional curvature and if the initial datum $u_0$  is bounded as well as its energy density. 
Later on, Liao and Tam \cite{Liao Tam} showed that if $M$ is a complete non-compact manifold, $N$ is compact with nonpositive sectional curvature, and if the initial map has finite total energy, then \eqref{15}  admits a  global solution which converges on compact subsets of $M$ to a harmonic map from $M$ into $N$. 
It is well known that if the sectional curvature  of $N$ is nonnegative, a blow up phenomenon may occur. Without this condition, one has to impose some assumptions on $N$ in order to prevent the blow up of the solution. We refer the reader to the work of Struwe (\cite{struwe1}, \cite{struwe2} ) and Chen-Struwe \cite{cs} concerning the singularities  of the harmonic heat flow. When $M$ is complete, Li and Wang \cite{Li Wang} proved that there exits a global solution of the harmonic heat flow from $M$ into a generalised regular ball of  the target manifold $N$ which converges at infinity to a harmonic map when the initial datum has finite energy (we  refer to \cite{Li Wang} for the notion of generalised regular balls.) \\

When $p\geq 2$ Fardoun-Regbaoui \cite{fr1} and Misawa \cite{Misawa1}  proved the global existence and the  convergence of the $p$-harmonic heat flow when $M$ and $N$ are compact and $N$ has nonpositive sectional curvature generalising the result Eells and Sampson \cite{Eells Sampson} to the case $p\ge 2$.   When $N$ has arbitrary sectional curvature,  Hungerbuhler \cite{ hu2} proved the existence of weak solutions in  the conformal case $p=m$. See also Hungerbuhler \cite{ hu1} when the target manifold is a  homogenuous space. For small initial data,  Fardoun-Regbaoui \cite{ fr2} obtained the existence and convergence of the flow. We also mention the recent work of Misawa \cite{ Misawa2} concerning the regularity of the $p$-harmonic heat flow. 

\bigskip

Our goal  in this paper is to extend the  results of Liao-Tam above  to the $p$-harmonic heat flow for $p\ge2$. Accordingly, we will introduce the notion of regular sets  inspired by the work of Li and Wang \cite{Li Wang} concerning the harmonic heat flow. 

\medskip

\begin{defi}
Let $\Omega$ be an open subset of a Riemannian manifold $N$ and let $\delta>0$.   We say that $\Omega$ is a $\delta$-regular set  if there exist a positive function $f\in C^2(\Omega)$ and a  constant  $C>0$ such that, for all $y \in \Omega$, we have 
\begin{equation} \label{c1}
\begin{cases} 
- \nabla^2 f(y) - K_2(y)f(y)h(y) \ge    \delta \frac{|\nabla f(y)|^2}{f(y)} h(y)   
\vspace{0.2cm} \cr
C^{-1} \le f(y) \le C,
\end{cases}
\end{equation}

\medskip

\noindent where $h$ is the metric of $N$ and  $K_2(y)=\sup\{K(y,\pi),0\}$, with  $K(y,\pi)$ being  the sectional curvature of a 2-plane $\pi\subset T_yN$. 
\end{defi}

\bigskip

According to our definition of $\delta$-regular sets,  any   Riemannian manifold with nonpositive  sectional curvature  is a $\delta$-regular set for any $\delta >0$. Indeed,  if $N$ has nonpositive sectional curvature, then  condition \eqref{c1} above is automatically satisfied by taking $f=1$. 

\bigskip

\begin{defi} 
We say that  $\Omega \subset N$  is a $\delta$-generalised regular ball if it is  a $\delta$-regular set  and if there exists a positive  function $f^* \in C^{2}(N)$ which is convex on $\Omega$  such that 
\begin{equation} \label{c2} 
\Omega=(f^*)^{-1}\big([0,a)\big)
\end{equation}
for some $a >0$. 
\end{defi}

\bigskip

\begin{exa}\label{exa1}

If $N$  is a Riemannian manifold with nonpositive sectional curvature, then $N$ is a $\delta$-generalised regular ball  for any $\delta >0$ by taking $f = f^*= 1$ and $a >1$. 

\end{exa}

\medskip

\begin{exa}\label{exa2}

On the sphere $\mathbb{S}^n$ any geodesic ball $B(y,r)$, with $0< r< {\pi\over 2}$, is a   $\delta$-generalised regular ball with $\delta= {(\cos r - \cos r_1) \cos r_1 \over \sin^2 r}$, where $r_1 $ is any real number such that $r < r_1 < {\pi\over 2}$. Indeed, in polar coordinates $(\rho, \theta)$  centered at $y$,  if we set $f(\rho, \theta) = \cos\rho -\cos r_1$  and $f^*$ any smooth function on $\mathbb{S}^n$ such that $f^*(\rho, \theta) = \rho^2$ on $B(y,r)$, then one can check that  $\nabla^2 f = -(\cos\rho)h$, where $h = d \rho^2 + \sin^2\hspace{-1mm}\rho\hspace{1mm} d\theta $ is the standard metric on $\mathbb{S}^n$, and that $f$ satisfies condition \eqref{c1} with $\delta = {(\cos r - \cos r_1) \cos r_1 \over \sin^2 r}$, and $f^*$ satisfies condition \eqref{c2}.  More generally, by using the Hessian comparaison theorem  on a  Riemanniann manifold $N$,  one can see that any regular geodesic ball $B(y, r)$ in the sense of Hildebrandt \cite{hi}   is a $\delta$-generalised regular  ball for some $\delta>0$ depending on $ r$. 

\end{exa}

\bigskip

Troughhout this paper we suppose $p \ge 2$.  We state our first main result : 

\medskip

\begin{theo}\label{th1}
Let $(M^m,g)$  and  $(N^n,h)$ be two  Riemannian manifolds such that $M$ is compact and $N$  is complete. Let  $\Omega \subset N$  be a $\delta$-generalised regular ball  with 
$$\delta > \delta_p := 3(p-2)^2\left(\sqrt{m}+ 2p+6\right)^2 + 3. $$
Then  for any $u_0\in C^{\infty}(M,\Omega)$,   there exists a  unique global solution $u$ of \eqref{15} such that  $u \in C_{loc}^{1+ \beta , \hspace{0.6mm}\beta/p}(M\times[0,\infty),\Omega)$  for some  $\beta\in(0,1)$. Moreover,  we have $\partial_tu \in L^2(M\times[0, +\infty))$  and satisfies  the energy inequality 
\begin{equation} \label{16}
 \int_{0}^{T}\int_{M}|\partial_t u(x,t)|^2  dx dt+E_p\big(u(. , T)\big)\leq E_p(u_0) 
\end{equation}
for all $T>0$. If  we assume in addition that $N$ is compact, then there exists a sequence $t_k\rightarrow\infty$ such that $u(.,t_k)$ converges in $C^{1+\beta'}(M,\Omega)$ (for all $\beta'<\beta$) to a $p$-harmonic map $u_\infty\in C^{1+\beta}(M,\Omega)$ satisfying
$E_{p}(u_\infty)\leq E_{p}(u_0)$.

\end{theo}

\bigskip

Theorem \ref{th1}  allows us to prove the following  result between complete Riemannian  manifolds which is, to our knowledge, the first result concerning  the existence of  $p$-harmonic heat flow from complete noncompact Riemannian manifolds.

\bigskip

\begin{theo}\label{th2}
Let $(M^m,g)$  and $(N^n,h)$ be two complete Riemannian manifolds.  Let $\Omega \subset N$  be a $\delta$-generalised regular ball  with 
$$\delta > \delta_p := 3(p-2)^2\left(\sqrt{m}+ 2p+6\right)^2 + 3. $$
Then  for any $u_0\in C^{\infty}(M,\Omega)$ with $E_p(u_0) < + \infty$,  there exists a global solution $u$ of \eqref{15} such that  $u \in C_{loc}^{1+ \beta , \hspace{0.6mm}\beta/p}(M\times[0,\infty),\Omega)$  for some  $\beta\in(0,1)$. Moreover, $\partial_tu \in L^2(M\times[0, +\infty))$  and satisfies  the energy inequality 
\begin{equation}\label{17}
 \int_{0}^{T}\int_{M}|\partial_t u(x,t)|^2  dx dt+E_p\big(u(. , T)\big)\leq E_p(u_0) 
\end{equation}
for all $T>0$. If we assume in addition that $N$ is compact, then there exists a sequence $t_k\rightarrow\infty$ such that $u(.,t_k)$ converges in $C_{loc}^{1+\beta'}(M,\Omega)$ (for all $\beta'<\beta$) to a $p$-harmonic map $u_\infty\in C_{loc}^{1+\beta}(M,\Omega)$ satisfying $E_{p}(u_\infty)\leq E_{p}(u_0)$.
\end{theo}

\bigskip

As a consequence of Theorem \ref{th2}, we have the following theorem concerning target manifolds with negative sectional curvature. It can be considered   as  a natural generalisation to the case $p \ge 2$ of the work of    Liao-Tam \cite{Liao Tam} concerning the heat flow of harmonic maps ($p=2$). 

\bigskip

\begin{theo} \label{th3}
Let $M$ and $N$ be two complete Riemannian manifolds such that  $N$ has nonpositive sectional curvature. Then for any $u_{0}\in C^{\infty}(M,N)$ with $E_p(u_0) < + \infty$,  there exists a global solution $u$ of \eqref{15} such that  $u \in C_{loc}^{1+ \beta , \hspace{0.6mm}\beta/p}(M\times[0,\infty),\Omega)$  for some  $\beta\in(0,1)$. Moreover, $\partial_tu \in L^2(M\times[0, +\infty))$  and satisfies  the energy inequality 
\begin{equation}\label{18}
 \int_{0}^{T}\int_{M} |\partial_t u(x,t)|^2  dx dt+E_p\big(u(. , T)\big)\leq E_p(u_0) 
\end{equation}
for all $T>0$. If we assume in addition that $N$ is compact, then there exists a sequence $t_k\rightarrow\infty$ such that $u(.,t_k)$ converges in $C_{loc}^{1+\beta'}(M,\Omega)$ (for all $\beta'<\beta$) to a $p$-harmonic map $u_\infty\in C_{loc}^{1+\beta}(M,\Omega)$ satisfying $E_{p}(u_\infty)\leq E_{p}(u_0)$.
\end{theo}

\bigskip

Our method allows us to prove  the following Liouville Theorem for $p$-harmonic maps.

\medskip

\begin{theo}\label{th4}
Let $M$ be a complete Riemannian manifold with nonnegative Ricci curvature and $N$ be a complete Riemannian manifold.  Suppose that $\Omega\subset N$ is a $\delta$-regular set with $\delta >\delta_p$, where $\delta_p$ is as in Theorem \ref{th1}. If $u \in C_{loc}^{1}(M, \Omega)$  is a $p$-harmonic map from $M$ into $\Omega$ with finite $p$-energy, then $u$ is constant. In particular, if $N$ has nonpositive sectional curvature, then any $p$-harmonic map $u \in C_{loc}^{1}(M, N)$  with finite $p$-energy from $M$ to $N$  is constant. 
\end{theo}

\bigskip

The paper is organised as follows.  The heat flow equation being  a degenerate parabolic problem, we first establish  in Section \ref{s2} the existence of a global solution to the regularised equation of \eqref{15}. We then prove   uniform {\it a priori} gradient estimates on the solutions of the regularised equation in Section \ref{s3}. Section \ref{s4}  is devoted to the proof of our main results.


\bigskip

\section{The Regularised Heat Flow}\label{s2}

\bigskip

Since  \eqref{15} is a degenerate parabolic system, one can not apply directly the existence theory for parabolic equations. To overcome this difficulty  we introduce the regularised  $p$-harmonic heat flow equation.  Namely, for $0<\varepsilon<1$, the regularised  $p$-energy of $u$ is defined by
$$E_{p,\varepsilon}(u)=\frac{1}{p}\int_M\left(|\nabla u|^2+\varepsilon\right)^\frac{p}{2}dx, $$
and the  gradient flow associated to $E_{p,\varepsilon}$ is given by the following  second order parabolic system 
  \begin{equation}\label{21}
    \left\{
                \begin{array}{ll}
                  \partial_t u-\Delta_{p,\varepsilon}=\left(|\nabla u|^2+\varepsilon\right)^{\frac{p-2}{2}}A(u)(\nabla u,\nabla u),\\
                  \\
               
                 u(x,0)=u_{0}(x)\\
            
                \end{array}
              \right.   
\end{equation}
where  $$\Delta_{p,\varepsilon}=\frac{1}{\sqrt{|g|}}\partial_i\Big(\sqrt{|g|}\left( |\nabla u|^2+\varepsilon\right)^{\frac{p-2}{2}}g^{ij}\partial_j u\Big)$$ 
is the regularised $p$-Laplacian of $M$.

\bigskip

Since \eqref{21} is a parabolic system, then it follows from the classical  theory of parabolic equations that \eqref{21} admits a unique smooth solution $u_{\varepsilon}$ defined on a maximum interval $[0, T_{\varepsilon})$. For the sake of simplicity we denote $u$ our solution instead of  $u_{\varepsilon}$ and $T$ instead of  $T_{\varepsilon}$.  We have the following proposition.

\begin{prop}\label{p21}Let $p \ge 2$ and let  $\Omega \subset N$  be a $\delta$-generalised regular ball  for some $\delta >0$. Let $u_{0}\in C^{\infty}(M,\Omega)$ and let $u$ be the solution of the regularised problem \eqref{21} defined on a maximal interval $[0, T)$. Then we have  for any $(x,t) \in M \times[0, T)$ 
\begin{equation}\label{22}
u(x,t) \in \Omega. 
     \end{equation}
Moreover, $u$ satisfies  the energy formula 
\begin{equation}\label{23}
 {d \over dt} E_{p,\varepsilon}\big(u( . , t)\big)  =  - \int_M|\partial_t u(x,t)|^2  \ dx. 
     \end{equation}
In particular the energy $E_{p,\varepsilon}$ is nonincreasing along the flow.
\end{prop}

\medskip

\begin{proof}  Since $\Omega$ is a generalised regular ball, then there exist a positive function $f^* \in C^2(N)$ which is convex on  $\Omega$  and $a>0$ such that $\Omega=(f^*)^{-1}([0,a))$.  Let 
$$T_1 = \sup \big\{ t \in [0, T) \ : \  u\left(M\times[0,t]\right) \subset \Omega \big\} $$
and suppose by contradiction that $T_1 < T$.  Since $u_0(M) \subset \Omega$ and $M$ is compact, then by continuity of $u$ we have  that $ T_1 >0$. Then we compute on $M \times [0, T_1)$ 
\begin{align*}
    \partial_t(f^*\circ u)-\hbox{div}\left((|\nabla u|^2+\epsilon)^{p-2\over 2}\nabla f^*\circ u\right)&=\Big\langle(\nabla f^*)\circ u, \left(\partial_t u- \Delta_{p,\epsilon}u\right)\Big\rangle\\
    &- (\nabla u|^2+\epsilon)^{p-2\over 2}(\nabla^2f^*)\circ u \big(\nabla u,\nabla u\big)\\
    &=  \left(|\nabla u|^2+\varepsilon\right)^\frac{p-2}{2}\big\langle (\nabla f^*)\circ u, A(u)(\nabla u,\nabla u)\big\rangle\\
    &-\left(|\nabla u|^2+\varepsilon\right)^\frac{p-2}{2}(\nabla^2f^*)\circ u(\nabla u,\nabla u), 
\end{align*}
which implies, since $\nabla f^*$ is orthogonal to $A(u)(\nabla u,\nabla u)$ and $f^*$ is convex on $\Omega$,  that 
\begin{equation*}
\partial_t(f^*(u))- \hbox{div}  \left(\left(|\nabla u|^2+\varepsilon\right)^{\frac{p-2}{2}} \nabla f^*(u) \right) \le 0.
\end{equation*} 
 Hence it follows  from the maximum principle for parabolic equations that for all $t\in[0,T_1)$, we have
$$\max_{x\in M} f^*(u(x,t))\leq \max_{x\in M} f^*(u_0(x))$$
which implies  by continuity of $f^*$  and $u$  that 
$$\max_{x\in M} f^*(u(x,T_1))\leq \max_{x\in M} f^*(u_0(x)).$$
Since $u_0(M)\subset\Omega$, then $ \displaystyle  \max_{x\in M}f^*(u_0(x))<a$, so 
$$ \max_{x\in M}f^*(u(x,T_1))< a. $$
 It follows  by continuity  of $f^*(u)$ that there exists $\alpha >0$ such  $\displaystyle  \max_{x\in M}f^*(u(x,t))< a$ for all  $t \in [0, T_1+ \alpha]$, that is, $\displaystyle u(M, t) \subset \Omega $ for all  $t \in [0, T_1 + \alpha]$ contradicting the definition of $T_1$. 

Now to prove \eqref{23} it sufficies to  take the inner product (in $\R^L$) of equation \eqref{21} with $\partial_tu$ and integrate on $M$ to get 
$$\int_M|\partial_t u(x,t)|^2  \ dx = \int_M\left(|\nabla u|^2+\varepsilon\right)^\frac{p-2}{2}\big\langle \partial_tu, A(u)(\nabla u,\nabla u)\big\rangle dx  -  {d \over dt} E_{p,\varepsilon}\big(u( . , t)\big).$$
This achieves the proof of the proposition since $\big\langle \partial_tu, A(u)(\nabla u,\nabla u)\big\rangle = 0$.

\end{proof}

\bigskip

In order to prove uniform gradient estimates on the solution of the regularised equation \eqref{21} we need  a Bochner-type formula on $u$. To this end let us introduce the following notations. We set for all  $0< \varepsilon< 1$: 
$$F = |\nabla u|^2 + \varepsilon , $$
and let $L_p$ be the operator defined  for $\varphi \in C^{2}(M)$ by 
\begin{equation}\label{24}
   L_p(\varphi) = \hbox{div}\left(F^{p-2\over 2} \nabla\varphi\right).
     \end{equation}
We define the symmetric  contravariant 2-tensor $B$  on  $M$ by setting in local coordinates on $M$ :
\begin{equation}\label{25}
  B_{ij}= {\langle \partial_ku, \partial_lu \rangle \over F}g^{ik}g^{lj}.
     \end{equation}
One checks immediately that  for any   $x \in M$ and  covectors $X, Y \in T^*_xM$, we have 
 \begin{equation} \label{250}
 \begin{cases}
  B(X, X) \ge 0   \cr
  
  B(X, Y) \le |X| |Y|.
  \end{cases}
 \end{equation} 
 
 \bigskip
 
If  $X= X_i dx^i \in T^*_xM$,  we denote by  $B(X, .)$  the vector in $ T_xM$ defined by $B(X, .)= B_{ij}X_i \partial_j$. 
 
 \medskip

  \noindent Then we have the following Bochner-type formula 
  
\begin{equation}\label{26}
    \begin{split}
        \partial_t F- L_{p}F&=  (p-2) \hbox{div}\left(F^{p-2\over 2} B(dF, .)\right) -2F^{\frac{p-2}{2}}|\nabla^2 u|^2-{(p-2)\over 2}F^{\frac{p-4}{2}}|\nabla F|^2  \\
       & - 2F^{\frac{p-2}{2}}\hbox{Ric}^M \big(\nabla u , \nabla u\big)
       +2 F^{\frac{p-2}{2}}\big\langle \hbox{Riem}^N(\nabla u ,\nabla u )\nabla u, \nabla u\big\rangle, 
    \end{split}
\end{equation}
\medskip

\noindent  where  $\hbox{Ric}^M$ is the Ricci tensor  of  $M$  and $\hbox{Riem}^N$ is the Riemann curvature tensor of $N$ with the following notations  in an orthonormal frame $\{ e_1 , \cdots, e_m \}$ of  $T_xM$  :

$$\hbox{Ric}^M \big(\nabla u , \nabla u\big)= \sum_{k=1}^L\sum_{i=1}^m\hbox{Ric}^M \big(\nabla_{e_i} u^k ,  \nabla_{e_i} u^k\big) $$
and 
$$ \big\langle \hbox{Riem}^N(\nabla u ,\nabla u )\nabla u, \nabla u\big\rangle = \sum_{i,j=1}^m
 \big\langle \hbox{Riem}^N(\nabla_{e_i} u ,\nabla_{e_j} u )\nabla _{e_i}u, \nabla_{e_j} u\big\rangle $$

\bigskip


\section{Gradient Estimates}\label{s3}

\medskip

In this section, we derive uniform gradient estimates on  the solution  $u$ of the regularised equation \eqref{21}. We first need the following usefull inequality.

\medskip

\begin{prop}\label{p31}
 Let $(M^m,g)$,  $(N^n,h)$ be two  Riemannian manifolds and let  $\Omega \subset N$  be a $\delta$-regular set.  Let  $u :  M\times [0, T) \to N$  be a smooth solution of \eqref{21}having its image in  $\Omega$  and  set 
$$\varphi(x,t)=\frac{F(x,t)}{f^2(u(x,t))},$$ 
 where $F(x,t)= |\nabla u(x,t)|^2 + \varepsilon$ and $f$ is the function satisfying condition \eqref{c1}.  Then we have  at any point $(x,t) \in  M\times [0, T)$ 
\begin{equation}\label{30}
\begin{split}
    & \partial_t\varphi-L_{p}\varphi   \leq  (p-2)\hbox{div}\left(F^{p-2 \over 2} B(d\varphi, .) \right)   - {1\over 40} F^{{p-4\over 2}}(f\circ u)^2|\nabla \varphi|^2  \\
     &- 2(\delta-\delta_p) {F^{{p\over 2}}\over (f\circ u)^4}\left|(\nabla f)\circ u \right|^2 |\nabla u|^2 + \  2 K_1 {F^{{p-2\over 2}}\over (f\circ u)^2} |\nabla u|^2, 
     \end{split}
\end{equation}
where $\delta_p =   3(p-2)^2\Big( \sqrt{m} + 2p+ 2  \Big)^2 + 3$   \ and \  $-K_1 \le 0$  is a lower bound of the Ricci curvature of $M$ at $x$.  The operator $L_p$ and the tensor $B$ are defined by \eqref{24} and \eqref{25} above. 
\end{prop}

\begin{proof}
Fix a point $ x_0  \in M$, then  in normal coordinates at   $x_0$, a  basic computation gives,
\begin{equation} \label{31}
\begin{split} 
    \partial_t\varphi-L_{p}\varphi =  \ &  \frac{1}{(f\circ u)^2}\Big(\partial_t F-L_{p} F\Big)- 2 \frac{F}{(f\circ u)^3}\Bigl(\partial_t (f\circ u)-L_{p}(f\circ u)\Bigr) \\
     + & 4\frac{F^{\frac{p-2}{2}}}{(f\circ u)^3} \nabla F\cdot\nabla (f\circ u)  - 6\frac{F^\frac{p}{2}}{(f\circ u)^4} |\nabla(f\circ u)|^2,
        \end{split}
\end{equation}  

\noindent  where the dot  \ $\cdot$ \  denotes the Riemannian  inner product on $M$.  By  using Bochner's formula \eqref{26},  the first  term in the right hand side of \eqref{31} can be bounded as : 

\begin{equation} \label{32}
\begin{split} 
 &\frac{1}{(f\circ u)^2}\Big(\partial_t F-L_{p} F\Big)  \leq   {(p-2)\over (f\circ u)^2}  \hbox{div}\left(F^{p-2 \over 2} B(dF, .) \right)    +  2K_1 {F^{\frac{p-2}{2}}|\nabla u|^2\over (f\circ u)^2}   \\
 &   +  2K_2{F^{\frac{p}{2}}|\nabla u|^2 \over (f\circ u)^2} -2{ F^{\frac{p-2}{2}}\over (f\circ u)^2} |\nabla^2 u|^{2} - {1\over2}(p-2 ){F^{\frac{p-4}{2}}\over (f\circ u)^2} \left|\nabla F \right|^2, 
\end{split}
\end{equation}
where  $ - K_1 \le 0 $ is a lower bound of the Ricci curvature of $M$, and   $K_2 \ge 0 $ is an upper bound of the sectional curvature of $N$.   To bound the second term in the right hand side of \eqref{31}, a direct computation gives 

\begin{equation*}
    \begin{split}
  \partial_t(f\circ u)-L_{p}(f\circ u)=\Big\langle(\nabla f)\circ u, \left(\partial_t u- L_{p}u\right)\Big\rangle - F^{\frac{p-2}{2}}(\nabla^2f)\circ u \big(\nabla u,\nabla u\big)
     \end{split}
    \end{equation*}
 where  in local coordinates : 
 $$ (\nabla^2f)\circ u \big(\nabla u,\nabla u\big)= \sum_{i,j=1}^mg^{ij}(\nabla^2f)\circ u \big(\partial_iu,\partial_ju\big) $$
 and where $\langle \cdot , \cdot \rangle $ denotes the inner product of $\mathbb{R}^L$ (we recall here that $N$ is isometrically embedded in $\mathbb{R}^L$).  Since  $\big\langle(\nabla f)\circ u, \left(\partial_t u- L_{p}u\right)\big\rangle  = 0 $ by equation \eqref{21}  ( $\partial_t u-\Delta_{p}u$ being  orthogonal to $T_{u}N$ and      
     $(\nabla f)\circ u \in T_{u}N$), then we obtain  
     
    \begin{equation} \label{33}
  \partial_t(f\circ u)-L_{p}(f\circ u)    =-F^{\frac{p-2}{2}}(\nabla^2f)\circ u \big(\nabla u,\nabla u\big).
    \end{equation}
    
    \medskip
    
\noindent Substituting \eqref{33}  and \eqref{32} in \eqref{31} gives

\begin{equation} \label{34}
    \begin{split}
   \partial_t\varphi-L_{p}\varphi & \leq  {(p-2)\over (f\circ u)^2}  \hbox{div}\left(F^{p-2 \over 2} B(dF, .) \right)
   -2{ F^{\frac{p-2}{2}}\over (f\circ u)^2} |\nabla^2 u|^{2}- {1\over2}(p-2 ){F^{\frac{p-4}{2}}\over (f\circ u)^2} \left|\nabla F \right|^2\\
   &+ 2K_1 {F^{\frac{p-2}{2}} \over (f\circ u)^2} |\nabla u|^2  
  +  2  \frac{F^{\frac{p}{2}}}{(f\circ u)^3}\bigg(K_2(f\circ u) |\nabla u|^2 +(\nabla^2f)\circ u \big(\nabla u,\nabla u\big)\bigg) \\
      &  + 4\frac{F^{\frac{p-2}{2}}}{(f\circ u)^3} \nabla F\cdot\nabla(f\circ u)  - 6\frac{F^\frac{p}{2}}{(f\circ u)^4} |\nabla(f\circ u)|^2 .
    \end{split} 
    \end{equation}   
    
Since $f$ satisfies condition \eqref{c1}, then we have

\begin{equation*}
K_2 (f\circ u) |\nabla u|^2 +(\nabla^2f)\circ u \big(\nabla u ,\nabla u\big)\le  -\delta{ \big|(\nabla f)\circ u\big|^2  \over (f\circ u)}|\nabla u|^2,
\end{equation*}
\noindent so it follows from \eqref{34} 
 
\begin{equation} \label{35}
  \begin{split}
   \partial_t\varphi-L_{p}\varphi & \leq  {(p-2)\over (f\circ u)^2}  \hbox{div}\left(F^{p-2 \over 2}  B(dF, .) \right)  -2{ F^{\frac{p-2}{2}}\over (f\circ u)^2} |\nabla^2 u|^{2}- {1\over2}(p-2 ){F^{\frac{p-4}{2}}\over (f\circ u)^2} \left|\nabla F \right|^2   \\
   & + 2K_1 {F^{\frac{p-2}{2}} \over (f\circ u)^2}|\nabla u|^2   -2\delta  \frac{F^{\frac{p}{2}}}{(f\circ u)^4}\left|(\nabla f)\circ u\right|^2 |\nabla u|^2 \\
   &  + 4\frac{F^{\frac{p-2}{2}}}{(f\circ u)^3} \nabla F\cdot\nabla(f\circ u)  - 6\frac{F^\frac{p}{2}}{(f\circ u)^4} \left|(\nabla f)\circ u\right|^2 .
    \end{split}
\end{equation}

To estimate the first term  in the right hand side of  \eqref{35} we compute, by using the fact  that $F= (f\circ u)^2\varphi$ and that we are working in normal coordinates at $x_0$, 

\begin{equation} \label{36}
\begin{split}
{1\over (f\circ u)^2}\hbox{div}\left(F^{p-2 \over 2}  B(dF, .) \right)  &=  \hbox{div}\left(F^{p-2 \over 2}  B(d\varphi, .) \right) + 2 \hbox{div}\left({F^{p \over 2}\over (f\circ u)^3}  B(d(f\circ u), .) \right) \\
& + 2 {F^{p-2 \over 2}\over (f\circ u)^3} B\left(dF, d(f\circ u)\right)  \\
& =  \hbox{div}\left(F^{p-2 \over 2} B( d\varphi, . )\right) +  2\partial_i\bigg({F^{\frac{p-2}{2}}\over (f\circ u)^3} \langle\partial_i u ,\partial_ju\rangle  \partial_j(f\circ u) \bigg) \\
& + 4 {F^{\frac{p-4}{2}}\over (f\circ u)^3} \langle\partial_i u ,\partial_ju\rangle \partial_iF \partial_j(f\circ u) \\
&=  \hbox{div}\left(F^{p-2 \over 2} B(d\varphi, . )\right) + (p+2){F^{\frac{p-4}{2}}\over (f\circ u)^3} \langle\partial_i u ,\partial_ju\rangle \partial_iF \partial_j(f\circ u) \\
&+ 2 {F^{\frac{p-2}{2}}\over (f\circ u)^3}\langle\partial_{ii}^2u ,\partial_ju\rangle  \partial_j(f\circ u)+ 2 {F^{\frac{p-2}{2}}\over (f\circ u)^3}\langle\partial_{i}u ,\partial_{ij}^2u\rangle  \partial_j(f\circ u)  \\
& +   2 {F^{\frac{p-2}{2}}\over (f\circ u)^3} \langle\partial_{i} u ,\partial_ju\rangle (\nabla^2 f)\circ u \left(\partial_i u  ,\partial_ju\right) \\
&+  2 {F^{\frac{p-2}{2}}\over (f\circ u)^3} \langle\partial_{i}u ,\partial_{j}u\rangle \langle(\nabla f)\circ u ,\partial_{ij}^2u\rangle \\
& - 6 {F^{\frac{p-2}{2}}\over (f\circ u)^4}\langle\partial_i u ,\partial_ju\rangle   \partial_i(f\circ u)\partial_j(f\circ u)
\end{split}
\end{equation} 

\medskip

 Since by condition \eqref{c1} $f$ is concave, then we have 
 \begin{equation} \label{37}
 \langle\partial_{i} u ,\partial_ju\rangle (\nabla^2 f)\circ u \left(\partial_i u  ,\partial_ju\right) \le 0.
 \end{equation} 
 
 To bound the other terms in \eqref{36} observe that the last term is nonpositive  and the other terms can be bounded by using  the Cauchy-Schwarz inequality. So we obtain from \eqref{36}  and \eqref{37}  
 
 \begin{equation*}
 \begin{split}
 &{1\over (f\circ u)^2}\hbox{div}\left(F^{p-2 \over 2} B(dF, .) \right)     \le  \   \hbox{div}\left(F^{p-2 \over 2} B(d\varphi, .) \right)  +(p+2) {F^{p-4\over2} \over (f\circ u)^3} \left|\nabla(f\circ u)\right| |\nabla u|^2 \left|\nabla F \right|  \\
& + 2 \left(\sqrt{m} + 1\right) {F^{p-2\over 2}\over (f\circ u)^3} \left|\nabla(f\circ u)\right| |\nabla u| \left|\nabla^2u\right|+ 2 {F^{p-2\over 2}\over (f\circ u)^3} \left|(\nabla f)\circ u\right| |\nabla u|^2 \left|\nabla^2u\right|
 \end{split}
 \end{equation*} 

\medskip

\noindent which gives,   by Young's inequality for $\alpha > 0$ (to be chosen later),

\begin{equation}\label{38}
    \begin{split}
&{(p-2)\over (f\circ u)^2} \hbox{div}\left(F^{p-2 \over 2} B(dF, . )\right)    \le (p-2)  \hbox{div}\left(F^{p-2 \over 2} B(d\varphi, . )\right)+ {\alpha\over 2}(p+2)(p-2) {F^{{p-4\over 2}} \over (f\circ u)^2}\left|\nabla F \right|^2   \\
&   +    {1\over 2\alpha}(p+2)(p-2)\frac{F^{{p\over 2}}}{(f\circ u)^4} \left|\nabla(f\circ u)\right|^2    
 +  \alpha(p-2) \left(\sqrt{m} + 2\right)  {F^{{p-2\over 2}}\over (f\circ u)^2} \left|\nabla^2u\right|^2\\
 &+ {1\over \alpha}(p-2) \left(\sqrt{m} + 1\right) {F^{{p\over 2}} \over (f\circ u)^4} \left|\nabla(f\circ u)\right|^2  
+ {p-2\over \alpha}{ F^{{p\over 2}} \over (f\circ u)^4} |(\nabla f)\circ u|^2|\nabla u|^2 ,
 \end{split} 
\end{equation}
where we have used the fact that $F= |\nabla u|^2 + \varepsilon \ge |\nabla u|^2$. 

\medskip

Now, to bound the terms \  $ 4\dfrac{F^{\frac{p-2}{2}}}{(f\circ u)^3} \nabla F\cdot\nabla(f\circ u)  - 6\dfrac{F^\frac{p}{2}}{(f\circ u)^4} \left|\nabla(f\circ u)\right|^2$ in \eqref{35},    we use again  Young's inequality  for any $\beta >0$ : 

 \begin{equation}  
 \begin{split} \label{39}
4\frac{F^{\frac{p-2}{2}}}{(f\circ u)^3} \nabla F\cdot\nabla(f\circ u)  - 6\frac{F^\frac{p}{2}}{(f\circ u)^4} \left|\nabla(f\circ u)\right|^2   \le   2\beta \frac{F^{\frac{p-4}{2}}}{(f\circ u)^2} |\nabla F|^2    + \left({2 \over \beta} - 6 \right)\frac{F^\frac{p}{2}}{(f\circ u)^4} \left|\nabla(f\circ u)\right|^2. 
  \end{split}
 \end{equation}

By combining \eqref{35}, \eqref{38} and \eqref{39},  we  obtain  
\begin{equation} \label{040}
\begin{split}
 & \partial_t\varphi-L_{p}\varphi  \leq (p-2)\hbox{div}\left(F^{p-2 \over 2} B(d\varphi , .) \right)  +  \Big(-2 +   \alpha(p-2)\left(\sqrt{m} +2\right)\Big)  \frac{F^{\frac{p-2}{2}}}{(f\circ u)^2}|\nabla^2 u|^2 \\
  &+ \left(-{1\over 2}(p-2) + {\alpha \over 2}(p+2)(p-2)  + 2\beta \right) {F^{{p-4\over 2}} \over (f\circ u)^2}|\nabla F|^2  +  2 K_1{F^{p-2\over2}  \over (f\circ u)^2} |\nabla u|^2 \\
  & + \left( -6 +{2\over \beta} +      {1\over \alpha}(p-2) \left({p\over 2} +  \sqrt{m} +2 \right)\right) {F^{{p\over 2}}  \over (f\circ u)^4} \left|\nabla(f\circ u)\right|^2 \\
  & +  \left(  -2\delta + {p-2\over \alpha} \right) {F^{{p\over 2}} \over (f\circ u)^4} |(\nabla f)\circ u|^2 |\nabla u|^2.
  \end{split}
\end{equation} 

\medskip

 Observe that $|\nabla F|^2= |\nabla(|\nabla u|^2)|^2 \le 4 |\nabla^2u|^2|\nabla u|^2$. Hence if we  choose $\beta = {p-1 \over 5}$ and fix $\alpha >0$ such that  
 
 \begin{equation} \label{041}
  -2 +  \alpha(p-2)\left(\sqrt{m} +2\right) \le 0,
 \end{equation}
\noindent  then it follows from \eqref{040} that 

\begin{equation}\label{042}
\begin{split}
 & \partial_t\varphi-L_{p}\varphi \leq  (p-2)\hbox{div}\left(F^{p-2 \over 2} B(d\varphi, . ) \right) +  2K_1 {F^{p-2\over2} \over (f\circ u)^2}|\nabla u|^2  \\
 & - \left( {p-1\over 10}  -  {1\over 4}\alpha(p-2)\Bigl(\sqrt{m} + 2p+6\Bigr)\right)  \frac{F^{\frac{p}{2}-2}}{(f\circ u)^2}|\nabla F|^2  \\
    &  \left( - 6 + {10 \over p-1} + {1\over \alpha}(p-2) \left({p\over 2}+  \sqrt{m}  + 2 \right)\right){F^{{p\over 2}} \over (f\circ u)^4}\left|\nabla(f\circ u)\right|^2  \\
      & +  \left(  -2\delta + {p-2\over \alpha} \right) {F^{{p\over 2}} \over (f\circ u)^4} |(\nabla f)\circ u|^2 |\nabla u|^2.
  \end{split}
\end{equation} 

\medskip

If we choose 
$$ \alpha = \begin{cases}  {1 \over 5(p-2)\left(\sqrt{m}+ 2p+2 \right)} \   \  \hbox{if} \  p>2 \cr
1 \hspace{2.6cm} \hbox{if} \ p= 2,
\end{cases}$$
 then it is clear that \eqref{041} is satisfied,   and observe that $\left|\nabla(f\circ u)\right| \le \left|(\nabla f)\circ u\right| |\nabla u|$, so it follows from \eqref{042} that 
\begin{equation}\label{043}
\begin{split}
 & \partial_t\varphi-L_{p}\varphi  \leq  (p-2)\hbox{div}\left(F^{p-2 \over 2} B(d \varphi, . )\right)      +   2K_1 {F^{p-2\over2} \over (f\circ u)^2}|\nabla u|^2   \\
  &  -{1\over 20}  \frac{F^{\frac{p-4}{2}}}{(f\circ u)^2}|\nabla F|^2 + (-2\delta + c_p)  {F^{{p\over 2}} \over (f\circ u)^4}\left|(\nabla f)\circ u\right|^2|\nabla u|^2  ,
   \end{split}
\end{equation} 
where 
$$  c_p =  -6 +{10 \over p-1} +  5 (p-2)^2 \left(\sqrt{m} +2p+6\right) \left( \sqrt{m} + {p\over 2} + 3\right).$$

\medskip

On the other hand we have $$\partial_iF=(f\circ u)^2 \partial_i\phi  + 2{\big\langle (\nabla f)\circ u , \partial_i u\big\rangle  \over f\circ u}, $$
so 
 $$|\nabla F|^2 \ge {1\over 2} (f\circ u)^4|\nabla\phi|^2 - 4 {F \over (f\circ u)^2}|\nabla(f\circ u) |^2.$$
 Hence it follows from \eqref{043}
 \begin{equation*} 
\begin{split}
  &\partial_t\varphi-L_{p}\varphi  \leq   (p-2)\hbox{div}\left(F^{p-2 \over 2} B(d\varphi, .) \right)   -{1\over 40}  F^{p-4 \over 2}(f\circ u)^2 |\nabla \varphi|^2  \\
  &+ (-2\delta + c_p +{1\over 5})  {F^{{p\over 2}} \over (f\circ u)^4}|(\nabla f)\circ u|^2|\nabla u|^2  + 2K_1 {F^{p-2\over2} \over (f\circ u)^2}|\nabla u|^2,
   \end{split}
\end{equation*} 
This proves the proposition since  ${1\over 2}c_p + {1\over 10}  \le \delta_p :=   3(p-2)^2\Big( \sqrt{m} + 2p + 6  \Big)^2 + 3$.

\end{proof}

\bigskip

Proposition \ref{p31} allows to prove the following important gradient estimates :

\medskip

\begin{prop}\label{p32} Let $(M^m,g)$  and $(N^n,h)$ be two complete Riemannian manifolds. Let $u : M \times [0, T) \to \Omega$ a smooth solution of \eqref{21}, where  $\Omega$ is a  $\delta$-regular set with 
$$\delta > \delta_p := 3(p-2)^2\left(\sqrt{m}+ 2p+2\right)^2 + 3. $$
Let $(x_0, t_0)  \in M\times(0,T)$ and $R>0$. Then if $t_0 >R$, we have 
$$\|\nabla u\|_{L^{\infty}\left(B(x_0,R/2)\times[t_0-R/2, t_0]\right)}  \le C_R\left(\int_{B(x_0, R) \times [t_0- R, t_0]} |\nabla u |^p dx dt + 1\right), $$
and if $t_0 \le R$, we have 
$$\|\nabla u\|_{L^{\infty}\left(B(x_0,R/2)\times[0, t_0]\right)}  \le C_R\left(\int_{B(x_0, R) \times [0,t_0]} |\nabla u|^p dx dt + \|\nabla u_0\|_{L^{\infty}(B(x_0, R))} + 1  \right),$$
where $C_R$ is a positive constant depending on $R,p, M$ and $\Omega$. 
\end{prop}

\begin{proof}
Fix $(x_0, t_0) \in M\times (0, T)$ and let $R >0$. In what follows $C_R$ is a positive constant that depends on $R, p, M$ and  $\Omega$  and its value may change from line to line. In this proof we suppose that $ t_0 > R$, as the case $t_0 \le R$ is easier to handle, and therefore, we omit it.  For $ 0 <  r< R$, we   set $Q_r= B(x_0,r)\times (t_0-r,t_0)$ where $B(x_0,r)$ is the geodesic ball of radius $r$. Let  $0 < \rho < r <R$ and let $\phi\in C^{1}_{0}\big(B(x_0,r)\times (t_0-r,\infty)\big)$ such that 
$\phi=1$ on  $Q_{\rho}$ with 
\begin{equation} \label{044}
0\leq \phi\leq 1, \ \ \ \ \ |\nabla\phi|\leq \frac{C_m}{r-\rho},  \ \ \ \ \ |\partial_t\phi|\leq \frac{C_m}{r-\rho},
\end{equation} 
where $C_m$ is a positive constant depending only on  the dimension $m$ of $M$. 

\medskip

As in Proposition \ref{p31}, let $\varphi= \dfrac{F} { f^2(u) } $, where $F = |\nabla u|^2 + \varepsilon$ and $f$ satisfies condition \eqref{c1}.  
If we multiply inequality \eqref{30} by $\varphi^\gamma\phi^2$, where $\gamma\geq 0$, and we integrate on $Q_r$, by using the hypothesis that $\delta \ge \delta_p$,  we get  
\begin{equation} \label{045}
    \begin{aligned}
     &\frac{1}{\gamma+1}\sup_{t\leq t_0}\int_{B(x_0,r)}\varphi^{\gamma+1} \ \phi^2 \ dx-\frac{2}{\gamma+1}\int_{Q_{r}}\varphi^{\gamma+1}\phi\partial_t\phi \ dx dt\\
     &+\gamma \int_{Q_r} f^{p-2}(u)\phi^2\varphi^{\frac{p}{2}+\gamma-2} |\nabla\varphi|^2\ dx dt
     - 2\int_{Q_{r}}f^{p-2}(u)\varphi^{\frac{p}{2}+\gamma-1}\phi |\nabla\phi| |\nabla\varphi| \ dx dt \\
     &\leq -\frac{1}{40}\int_{Q_r}  f^{p-2}(u)\phi^2 \varphi^{\frac{p}{2}+\gamma-2}|\nabla\varphi|^2 \ dx dt+ 2K_R \int_{Q_r}  f^{p-2}(u)\phi^2 \varphi^{\frac{p}{2}+\gamma} \ dx dt\\
      &- (p-2)\gamma \int_{Q_r}f^{p-2}(u)\phi^2 \varphi^{\frac{p}{2}+\gamma-2}B(d\varphi,d\varphi)  \ dx dt  \\
      & -2(p-2)\int_{Q_r} f^{p-2}(u)\phi \varphi^{\frac{p}{2}+\gamma-1} B(d\varphi, d\phi)  \ dx  dt, \\
    \end{aligned}
\end{equation}
where $ - K_R \le 0$ is a lower bound of the Ricci curvature of $M$ on $B(x_0, R)$ and $B$ is given by \eqref{25}. We  have by \eqref{250} that 
\begin{equation} \label{046}
B(d\varphi, d\varphi) \ge 0 \  \  \hbox{and} \  \   \left|B(d\varphi, d\phi)\right| \le \left|\nabla\varphi\right| \left|\nabla\phi\right| .
 \end{equation}

It follows from \eqref{044}, \eqref{045} and \eqref{046} that 
\begin{equation*}
   \begin{aligned} 
 &\frac{1}{\gamma+1}\sup_{t\leq t_0}\int_{B(x_0,r)}\varphi^{\gamma+1} \ \phi^2 \ dx + \left(\gamma + {1\over 40}\right) \int_{Q_r} f^{p-2}(u)\phi^2\varphi^{\frac{p}{2}+\gamma-2} |\nabla\varphi|^2\ dx dt \le \\
 &\frac{2}{(\gamma+1)(r-\rho)}\int_{Q_{r}}\varphi^{\gamma+1}\phi\ dx dt + 2K_R \int_{Q_r}  f^{p-2}(u)\phi^2 \varphi^{\frac{p}{2}+\gamma} \ dx dt \\
 &+ 2(p-1){1 \over r-\rho}\int_{Q_r}f^{p-2}(u)  \varphi^{\frac{p}{2}+\gamma-1}\phi |\nabla\varphi| \ dx  dt
 \end{aligned}
\end{equation*}
which gives,  by applying Young's inequality to the last term, 
\begin{equation}\label{047}
   \begin{aligned}
 &\frac{1}{\gamma+1}\sup_{t\leq t_0}\int_{B(x_0,r)}\varphi^{\gamma+1} \ \phi^2 \ dx + {1\over 2}\left(\gamma + {1\over 40}\right) \int_{Q_r} f^{p-2}(u)\phi^2\varphi^{\frac{p}{2}+\gamma-2} |\nabla\varphi|^2\ dx dt \le \\
 &\frac{2}{(\gamma+1)(r-\rho)}\int_{Q_{r}}\varphi^{\gamma+1}\phi\ dx dt + 2K_R \int_{Q_r}  f^{p-2}(u)\phi^2 \varphi^{\frac{p}{2}+\gamma} \ dx dt  \\
 &2\left(\gamma+ {1\over 40}\right)^{-1}{(p-1)^2 \over (r-\rho)^2}\int_{Q_r}  f^{p-2}(u)\varphi^{\frac{p}{2}+\gamma}  \ dx  dt.
 \end{aligned}
\end{equation}

\medskip

It is easy to see that
\begin{equation}\label{048}
    \begin{split}
        \int_{Q_r} f^{p-2}(u)\phi^2 \varphi^{\gamma+\frac{p}{2}-2} \ |\nabla\varphi|^2 \ dx dt&\geq \frac{8}{(p+2\gamma)^2}\int_{Q_r}|\nabla(\varphi^{\frac{\gamma}{2}+\frac{p}{4}}\phi)|^2f^{p-2}(u) \ dx dt\\
        &-\frac{16}{(p+2\gamma)^2}\int_{Q_r}\varphi^{\gamma+\frac{p}{2}}f^{p-2}(u) |\nabla\phi|^2 \ dx dt. 
    \end{split}
\end{equation}
Thus, if we multiply \eqref{048} by $\gamma+1$ and combine it with \eqref{047} and \eqref{044}, using the fact that $C^{-1} \le f \le C$, we get 
 
\begin{equation}\label{049}
    \begin{aligned}
     &\sup_{t\leq t_0}\int_{B(x_0,r)}\varphi^{\gamma+1}\phi^2 \ dx+\int_{Q_r}|\nabla(\varphi^{\frac{\gamma}{2}+\frac{p}{4}}\phi)|^2 \ dx dt \le \\
     & C_R \left( \left(\gamma+1 + \frac{1}{(r-\rho)^2} \right)\int_{Q_r}\varphi^{\frac{p}{2}+\gamma} \ dx dt +\frac{1}{r-\rho}\int_{Q_r}\varphi^{\gamma+1} dx dt \right). 
    \end{aligned}
\end{equation}

We recall the following Sobolev inequality for $m>2$ and $V\in C^{\infty}_{0}(B(x_0,R))$
$$\bigg(\int_{B(x_0,R)}V^{\frac{2m}{m-2}} \ dx\bigg)^{\frac{m-2}{2m}}\leq C_R\bigg(\int_{B(x_0,R)}|\nabla V|^2 \ dx\bigg)^{\frac{1}{2}}.$$
For $m=2$ we have  for any $s \ge 1$, 
$$\bigg(\int_{B(x_0,R)}V^s dx\bigg)^{1\over s}\leq C_{R,s}\bigg(\int_{B(x_0,R)}|\nabla V|^2 \ dx\bigg)^{\frac{1}{2}}.$$
In this proof we shall consider only the case $m>2$. The case $m=2$ can be  handled in the same way. 
Applying Sobolev  inequality to  $V=\varphi^{\frac{\gamma}{2}+\frac{p}{4}}\phi$  we get
\begin{equation}\label{491}
    \begin{split}
  \int_{t_0-r}^{t_0} \left(\int_{B(x_0,r)}(\varphi^{\frac{\gamma}{2}+\frac{p}{4}} \ \phi)^{\frac{2m}{m-2}} dx  \right)^{m-2\over m}  dt &\leq C_R \int_{Q_r}|\nabla(\varphi^{\frac{\gamma}{2}+\frac{p}{4}}\phi)|^2 \ dx dt.
    \end{split}
\end{equation}

On the other hand if we set  $q_{\gamma}=(1+\frac{2}{m})\gamma+\frac{p}{2}+\frac{2}{m}$, then  by using the fact that $\phi =1$ on $Q_{\rho}$, we infer from Holder's inequality that
\begin{equation}\label{050}
    \begin{split}
        & \int_{Q_{\rho}}\varphi^{q_{\gamma}} \ dx dt=\int_{Q_\rho} \phi^{4\over m}\varphi^{\frac{2\gamma}{m}+\frac{2}{m}} \ \phi^2\varphi^{\gamma+\frac{p}{2}} \ dx dt \leq \\
        &\sup_{t\leq t_0}\bigg(\int_{B(x_0,r)}\varphi^{\gamma+1} \ \phi^2 \ dx\bigg)^{2\over m}  \times \int_{t_0-r}^{t_0}\bigg(\int_{B(x_0,r)}\left(\varphi^{\frac{\gamma}{2}+\frac{p}{4}} \ \phi\right)^{\frac{2m}{m-2}} \ dx\bigg)^{\frac{m-2}{m}} \ dt.
    \end{split}
   \end{equation}

Hence by combining \eqref{049}, \eqref{491} and \eqref{050} we obtain 
\begin{equation}\label{051}
    \begin{split}
  & \int_{Q_{\rho}}\varphi^{q_{\gamma}} \ dx dt  \leq  \\
  & C_R   \left(  \left(\gamma+1 + \frac{1}{(r-\rho)^2} \right)\int_{Q_r}\varphi^{\frac{p}{2}+\gamma} \ dx dt +\frac{1}{r-\rho}\int_{Q_r}\varphi^{\gamma+1} dx dt \right)^{1+ {2\over m}}. 
    \end{split}
\end{equation}
By H\"older's inequality  and Young' inequality we have 
$${1 \over r-\rho}\int_{Q_r} \varphi^{\gamma+1}  \ dx \le  |Q_r| +  (r-\rho)^{-{2\gamma+ p \over 2\gamma + 2}}\int_{Q_r}\varphi^{\frac{p}{2}+\gamma} \ dx dt, $$
where $|Q_r|$ is the volume of $Q_r$. Since $|Q_r| \le |Q_R| \le C_R$, then it follows from \eqref{051} that 
\begin{equation}\label{052}
    \begin{split}
  & \int_{Q_{\rho}}\varphi^{q_{\gamma}} \ dx dt  \leq  \\
  & C_R  \left( \left(\gamma+ 1 + (r-\rho)^{-2} + (r-\rho)^{-{2\gamma+ p \over 2\gamma + 2}} \right)\int_{Q_r}\varphi^{\frac{p}{2}+\gamma} \ dx dt + 1 \right)^{1+ {2\over m}}. 
    \end{split}
\end{equation}

\medskip

We apply now Moser iteration process.  For $j \in \N$,  let $R_{j}=\dfrac{R(1+2^{-j})}{2}$ and $\theta=1+\frac{2}{m}$.  We define $\gamma_{j}=\theta^{j}-1$ and $a_{j}=\gamma_{j}+\frac{p}{2}$. Then we have
$$a_{j+1}=\theta \gamma_{j}+\frac{p}{2}+\frac{2}{m}= \theta a_{j} -{p-2\over m}.$$

If we set $\gamma= \gamma_{j}, \ r= R_j,  \  \rho= R_{j+1}$, then it is easy to check that 
$$\gamma+ 1 + (r-\rho)^{-2} + (r-\rho)^{-{2\gamma+ p \over 2\gamma + 2}} \le C_R \ 4^{pj} .$$

Thus it follows from \eqref{052} that 

$$ \int_{Q_{R_{j+1}}}\varphi^{a_{j+1}} \ dx dt  \leq  C_R \left(4^{pj}  \int_{Q_{R_{j}}}\varphi^{a_{j}} \ dx dt + 1 \right)^{\theta} $$

\medskip
\noindent which gives by setting $\displaystyle I_j =  \left(\int_{Q_{R_{j}}}\varphi^{a_j} \ dx dt + 1\right)^{\theta^{-j}}$, 
$$I_{j+1} \le C_R^{\theta^{-j-1}} 4^{pj \theta^{-j}}I_j.$$
Since $\displaystyle \sum_{j=0}^{\infty} j\theta^{-j} \le C$, then by iterating we get 
\begin{equation}\label{053}
 I_{j+1} \le C_R I_0
\end{equation}
Now observing that 
$$\left(\int_{Q_{R/2}} \varphi^{a_j} dx dt \right)^{1\over a_j} \le I_j^{\theta^j\over a_j}$$
 and using the fact that $\displaystyle \lim_{j\to + \infty} {\theta^j\over a_j}= 1$,  then it follows from \eqref{053} that 
$$ \|\varphi\|_{L^{\infty}(Q_{R/2})}  \le  C_R I_0= C_R \left(\int_{Q_R} \varphi^{p\over 2} dx dt  + 1\right) .$$
This proves the proposition since  $ \varphi = \dfrac{|\nabla u|^2 + \varepsilon} { f(u)}$ and $C^{-1} \le f \le C$.

\end{proof}

\medskip


\section{Global Existence and convergence}\label{s4}

\medskip

In this section we make use of our gradient estimates on the solution of the regularised $p$-harmonic flow  obtained in Section \ref{s3} to prove our main results. 

\bigskip

\begin{proof}[ \textbf{\textit{Proof of theorem \ref{th1}}}]

In this proof $C$ denotes a positive constant depending on $M, p, \Omega$ and the initial datum $u_0$, and whose value may change from line to line.  Let $u_{\varepsilon}$  be the solution of the regularised equation \eqref{21}  and let $[0, T_{\varepsilon})$  be its  maximal  existence interval. Since by Proposition \ref{p21}, $u_{\varepsilon}$ has its image in $\Omega$, then if we apply  Proposition \ref{p32} by taking $R=1$, using the compactness of $M$  and the fact that the $p$-energy functional $E_{p, \varepsilon}$ is nonincreasing along the flow (formula \eqref{23} in Proposition \ref{21}), we get 
\begin{equation} \label{40}
\|\nabla u_{\varepsilon}\|_{L^{\infty}(M\times[0, T_{\varepsilon}))} \le C. 
\end{equation}

\medskip

Suppose by contradiction that $T_{\varepsilon} < + \infty $. Then by integrating formula \eqref{23}  in Proposition \ref{21}, we have 

\begin{equation} \label{400}
\int_0^{T_{\varepsilon}}\int_{M}|\partial_t u_{\varepsilon}(x,t)|^2  dx dt  \le E_{p, \varepsilon}(u_0) \le E_{p, 1}(u_0) .
\end{equation}

On the other hand, we have   for all  $t \in [0, T_{\varepsilon})$ the following bound on the mean value $\overline{u}_{\varepsilon}(t)$ of $u_{\varepsilon}(. , t)$ : 
$$|\overline{u}_{\varepsilon}(t)| = {1\over |M|} \left| \int_{M} u_{\varepsilon}(x,t) dx\right|  \le  |\overline{u}_0| + {1\over |M|}\int_0^{T_{\varepsilon}}\int_M |\partial_tu_{\varepsilon}(x,t)| dx dt, $$
which implies by  using  the Cauchy-Schwarz inequality and \eqref{400} 

\begin{equation}\label{401}
|\overline{u}_{\varepsilon}(t)| \le |\overline{u}_0| + C\sqrt{T_{\varepsilon}} .
\end{equation}

\medskip

We have by the mean-value Theorem that 

\begin{equation}\label{402}
 \sup_{x\in M} |u_{\varepsilon}(x,t) -\overline{u}_{\varepsilon}(t)| \le    \hbox{diam}(M)\|\nabla u_{\varepsilon}(., t)\|_{L^{\infty}(M)}, 
 \end{equation}
 
\noindent where $\hbox{diam}(M)$ is the diameter of $M$.  Hence it follows from  \eqref{40}, \eqref{401} and \eqref{402} that 

\begin{equation} \label{403}
\|u_{\varepsilon}\|_{L^{\infty}(M\times[0, T_{\varepsilon}))} + \|\nabla u_{\varepsilon}\|_{L^{\infty}(M\times[0, T_{\varepsilon}))} \le C + C\sqrt{T_{\varepsilon}} .
\end{equation}

\bigskip

Using   \eqref{403} and the results of Dibenedetto \cite{Dibenedetto},  we have for some $\beta \in (0, 1)$, 
\begin{equation} \label{41}
    \|u_\varepsilon\|_{C^{1+ \beta,\hspace{0.3mm} \beta/p}(M\times[0,T_\varepsilon))} \leq C_{T_\varepsilon},
\end{equation}
where the constant $C_{T_\varepsilon}$ depends also on $T_{\varepsilon}$. The theory of linear parabolic equations (see \cite{Ladyzenskaya})  together with \eqref{41} give, for some $ 0 < \alpha < 1$, 
\begin{equation} \label{42}
 \|u_\varepsilon\|_{C^{2+\alpha, 1+ {\alpha\over 2}}(M\times [0,T_\varepsilon))} \leq  C_{T_\varepsilon}
\end{equation}
where $C_{T_\varepsilon}$ is a  new constant that also depends on the modulus of ellipticity $\varepsilon$. Estimate \eqref{42} implies that $u_{\varepsilon}$ can be extended beyond $T_{\varepsilon}$ contradicting  thus the maximality  of $T_\varepsilon$. Hence we have $T_\varepsilon =+\infty$.    \\

Now we are in position to prove Theorem \ref{th1}.  By the  result  above,  $u_{\varepsilon}$ is defined on $[0, +\infty)$ and we have   by  \eqref{41} for all $T > 0$ 

\begin{equation} \label{420}
    \|u_\varepsilon\|_{C^{1+ \beta,\hspace{0.3mm} \beta/p}(M\times[0,T])} \leq C_{T},
\end{equation}

\medskip

\noindent where $C_T$ is a positive constant depending on $T, u_0, p, M$ and $\Omega$ but not on $\varepsilon$. 

\bigskip

It follows from  estimate \eqref{420}  that there exist a sequence $\varepsilon_k\rightarrow 0$ and  a map $u\in C_{loc}^{1+\beta, \hspace{0.5mm}\beta/p}(M\times [0,+\infty),\Omega)$ such that
$ u_{\varepsilon_k}\rightarrow u$ in $C_{loc}^{1+\beta', \hspace{0.5mm} \beta'/p}(M\times[0,+\infty),\Omega) $ for all $\beta'<\beta$. 
In addition, the energy formula  \eqref{23}  in Proposition \ref{21} gives for all $T>0$, 
\begin{equation}
\int_{0}^{T}\int_{M}|\partial_t u_{\varepsilon_k}(x,t)|^2  dx  dt +E_{p, \varepsilon_k}\big(u_{\varepsilon_k}( . , T)\big)  \leq  E_{p, \varepsilon_k}(u_0) 
\end{equation}

\noindent which implies that  $\partial_tu_{\varepsilon_k}\rightarrow\partial_t u$ weakly in $L^2(M\times [0,+\infty))$ and we have the energy inequality  for the limit $u$ 
\begin{equation} \label{43}
\int_{0}^{T}\int_{M}|\partial_t u(x,t)|^2  dx  dt + E_p\big(u(. , T)\big) \le E_p\big(u_0\big).
\end{equation} 

Passing to the limit in \eqref{21} when $\varepsilon_k \rightarrow 0$,  one can easily check that $u$ is a solution of \eqref{15}.

\bigskip

In order to prove the uniqueness of solutions, we recall the following well known inequality valid for any $a, b \in \R^L$ and $p\ge 2$ :

\begin{equation} \label{433}
\big\langle |a|^{p-2}a - |b|^{p-2} b , a-b \big\rangle \ge |a-b|^p ,
\end{equation}

\medskip

\noindent where $\langle  . , . \rangle $ and $| .| $ denotes Euclidean inner product and the corresponding Eucidean norm in $\R^L$. 

 \medskip

 Let $T>0$ and let $u_1, u_2 \in C^{1+\beta}(M\times[0,T])$ be two solutions of \eqref{15} such that $u_1(., 0) = u_2(., 0)$. If we set $w = u_1 - u_2$, then taking the difference of the equations satisfied by $u_1$ and $u_2$ (the same as equation \eqref{15}), multiplying it by $w$, integrating on $M\times[0,T]$, and using \eqref{433} along with the fact that $\nabla u_1$ and $\nabla u_2$ are bounded in $L^{\infty}(M\times [0,T])$, one can easily check that for any $t \in [0, T]$, 

$$\int_M |w(x,t)| dx  \le C  \int_0^t\int_M |w(x,s)| dx ds. $$
The right hand side of the above inequality is increasing in $t$, therefore
\begin{equation*}
\sup_{t'\in[0,t]}\int_{M}|w(.,t')|^2\ dx \leq Ct\sup_{t'\in[0,t]}\int_{M}|w(.,t') \ dx,
\end{equation*}
thus, for $t<{1\over C}$ we get $w\equiv 0$ for $t'\in [0,t]$. Iterating the argument proves the assertion.
\bigskip

Now let us  prove the convergence of the flow at infinity when the target manifold $N$ is compact. First we observe that by the energy inequality \eqref{43} we have 
$$ \int_{0}^{+\infty}\int_{M}|\partial_t u(x,t)|^2  dx dt  \leq E_p\big(u_0\big),$$
which implies the existence of a sequence $t_k \to + \infty$ such that 
\begin{equation} \label{44}
\int_{M}|\partial_t u(x, t_k)|^2 dx \to 0   \  \  \hbox{as} \  \  t_k \to + \infty.
\end{equation}

On the other hand,  it follows from estimate \eqref{40} and the fact that  $N$ is compact 

$$   \|u\|_{L^{\infty}(M\times[0, +\infty))} + \|\nabla u \|_{L^{\infty}(M\times[0, +\infty))} \le C $$
and the results of Dibenedetto \cite{Dibenedetto} imply that 
\begin{equation} \label{45}
   \|u\|_{C^{1+ \beta,\hspace{0.5mm}\beta/p}(M\times[0,+\infty))} \leq C.
 \end{equation} 

Hence  by passing to a subsequence if necessary,  we deduce from \eqref{45} that  $(u(. , t_k))_k$ converges in $C^{1+ \beta'}(M, \Omega)$  for all $\beta' < \beta$  to a map $u_{\infty}  \in C^{1+ \beta}(M, \Omega)$. By passing to the limit in equation \eqref{15} and using \eqref{44} we have that $u_{\infty}$ is a $p$-harmonic map satisfying $E_p(u_{\infty}) \le E_p(u_0)$. The proof of Theorem \ref{th1} is complete. 

\end{proof} 

\bigskip

The proof of Theorem \ref{th2} relies on Theorem \ref{th1} by using an exhaustion of $M$ by a sequence of compact manifolds and the following proposition.

\medskip

\begin{prop} \label{p41} Let $u$ be the solution of problem \eqref{15} given by Theorem \ref{th1}. Then for any ball $B(x_0, R) \subset  M$,  there exists a constant $C_R >0$  depending on $B(x_0, R), p$ and $\Omega$   such that 
$$ \sup_{t\ge 0} \|\nabla u(. , t)\|_{L^{\infty}(B(x_0, R/2))}  \le C_R \left(\int_M |\nabla u_0|^p dx +  \|\nabla u_0\|_{L^{\infty}(B(x_0, R))}  +  1\right) . $$
\end{prop} 

\begin{proof} As in the proof of Theorem \ref{th1},  $u$ is the limit of sequence $(u_{\varepsilon_k})_{k}$ of solutions to the regularised problem \ref{21} such that $\varepsilon_k \to 0$. Then if we apply Proposition \ref{p32} to $u_{\varepsilon_k}$and pass to the limit when $k \to \infty$, we obtain  easily the desired result.

\end{proof} 

\bigskip

\begin{proof}[ \textbf{\textit{Proof of Theorem \ref{th2} }}]

Suppose that $(M, g)$ is a complete noncompact Riemannian manifold. Let $(U_i)_{i\ge1}$  be an exhaustion of $M$ by compact manifolds with smooth boundaries. More precisely, each $U_i$  is an open set of $M$ such that $\overline{U}_i$ is a compact manifold with smooth boundary $\partial U_i$ and 
\begin{equation}\label{460}
\begin{cases}  \overline{U}_i \subset U_{i+1} \cr
{\displaystyle \bigcup_{i\ge1}U_i = M}.
\end{cases}
\end{equation}

 \medskip

 In order to apply  Theorem \ref{th1} it is necessary to consider manifolds without boundary. To this end, we consider for each $i \ge1$,  the double manifold of $U_i$ that we denote by $\widetilde{U}_i$.  Thus $\widetilde{U}_i$ is a compact manifold without boundary such that $\overline{U}_i \subset \widetilde{U}_i$ and the metric $g$ on  $\overline{U}_i$ extends to  a $C^{1}$-metric  $\widetilde{g}_i$ on $\widetilde{U}_i$. We smooth out $\widetilde{g}_i$ on a neighborhood of $\partial U_i$. More precisely, for a fixed $ 0< \varepsilon < {1\over 4}\hbox{diam}(U_1)$, where $\hbox{diam}(U_1)$ is the diameter of $U_1$,  we let 
 
$$U_i^{\varepsilon} = \big\{ x \in U_i  \ : \  d\left(x, \partial U_i\right) > \varepsilon \big\}.$$

 Then the new metric, that we denote still $\widetilde{g}_i$  for simplicity,  is $C^{\infty}$ on $\widetilde{U}_i$, it can be chosen arbitrary close to $g$ in the $C^1$-norm and it satisfies 

\begin{equation}\label{46}
\widetilde{g}_i = g \  \  \hbox{on} \ \  U_i^{\varepsilon}.
\end{equation}

\medskip

In the same way, we extend the initial datum $u_0$ to a map  $ \widetilde{u}_{0,i}$ on $\widetilde{U}_i$ that we smooth out on a  neighborhood of $\partial U_i$, and one can choose $\widetilde{u}_{0,i}$ arbitrary close to $u_0$ in the $C^1$-norm.  Thus we have  $ \widetilde{u}_{0,i} \in C^{\infty}\left(\widetilde{U}_i, \Omega \right)$  with

\begin{equation} \label{47}
\widetilde{u}_{0,i}= u_0 \  \  \hbox{on} \ \  U_i^{\varepsilon}, 
\end{equation}

\medskip

\noindent and we may suppose without loss of generality that 
\begin{equation} \label{48}
\int_{\widetilde{U}_i} |\nabla \widetilde{u}_{0,i}|^p d\widetilde{g}_i \le  2 \int_{U_i} |\nabla u_0|^p dx+ 1,
\end{equation}
where $d\widetilde{g}_i $ is the volume element with respect to $\widetilde{g}_i$ and $dx$ is the volume element with respect to $g$.

Then we consider on each  $\widetilde{U}_i$ the $p$-harmonic heat flow problem 
\begin{equation} \label{49}
    \left\{
                \begin{array}{ll}
                  \partial_t u-\widetilde{\Delta}_{p, i} u =|\nabla u|^{p-2}A(u)(\nabla u,\nabla u),\\
                  \\
                 u(x,0)=\widetilde{u}_{0,i}(x), \\
            
                \end{array}
              \right. 
 \end{equation}
 where $\widetilde{\Delta}_{p, i}$ is the $p$-Laplacian with respect to the metric $\widetilde{g}_i$.  
 
 \medskip
 
 Thanks to Theorem \ref{th1},  problem \eqref{49} admits a global solution  $u_i \in C_{loc}^{1+ \beta, \hspace{0.3mm}\beta/p}\left( \widetilde{U}_i \times[0, +\infty), \Omega\right)$ satisfiying the $p$-energy inequality 
 \begin{equation} \label{50}
  \int_0^T \int_{\widetilde{U}_i} |\partial_t u_i|^2 d\widetilde{g}_i dt  + {1\over p} \int_{\widetilde{U}_i} |\nabla u_i(. , T) |^p d\widetilde{g}_i  \le {1\over p} \int_{\widetilde{U}_i} |\nabla \widetilde{u}_{0,i}|^p d\widetilde{g}_i 
  \end{equation}
 for all $T>0$. 
 
 \medskip

 Since $\widetilde{g}_i = g$ on $U_i^{\varepsilon}$, then $u_i$ is a solution of equation \eqref{15} in $U_i^{\varepsilon}$.   We shall prove uniform gradient estimates on $u_i$ on fixed balls of $M$. For each fixed $R>0$,  we denote by  $C_R$ a positive constant that depends on $R, M, p, \Omega$ and the initial datum $u_0$, and whose  value may change from line to line.  Fix $x_0 \in M$ and $R>0$, then we have from \eqref{460}  and the definition of $U_i^{\varepsilon}$, for $i$ large enough,  that $B(x_0, R) \subset U_i^{\varepsilon}$. It follows from  Proposition \ref{p41}  that 
 $$ \sup_{t\ge 0} \|\nabla u_i(. , t)\|_{L^{\infty}(B(x_0, R/2))}  \le C_R \left(\int_{\widetilde{U}_i}  |\nabla \widetilde{u}_{0,i}|^p dx +  \|\nabla \widetilde{u}_{0,i}\|_{L^{\infty}(B(x_0, R))}  +  1\right)  $$
 and by \eqref{48} we get 
 
\begin{equation} \label{51}
\begin{split}
\sup_{t\ge 0} \|\nabla u_i(. , t)\|_{L^{\infty}(B(x_0, R/2))}  & \le C_R \left(\int_{M}  |\nabla u_0|^p dx +  \|\nabla u_0\|_{L^{\infty}(B(x_0, R))}  +  1\right)\\
& \le C_R.
\end{split}
\end{equation}

As in the proof of Theorem \ref{th1}, by using \eqref{48}, \eqref{50}, \eqref{51}  and the mean value Theorem, we have for any $T >0$, 

\begin{equation} \label{52}
\|u_i\|_{L^{\infty}(B(x_0, R/2)\times [0, T])} \le C_R + C_R \sqrt{T}.
\end{equation} 

It follows from \eqref{51}, \eqref{52} and  the results of Dibenedetto \cite{Dibenedetto} on degenerate parabolic equations that 

 \begin{equation} \label{53}
 \|u_i\|_{C^{1+ \beta,\hspace{0.3mm}\beta/p}(B(x_0, R/2) \times [0, T])} \le C_{R,T},
 \end{equation} 
 
 \medskip
 
\noindent  for some constant $\beta \in (0, 1)$, where the constant $C_{R,T}$ depends also on $T$. 

\medskip

Since $B(x_0, R) \subset U_i^{\varepsilon}$ for $i$ large enough and $\widetilde{g}_i = g$ on $U_i^{\varepsilon}$, then we have from \eqref{50}  and \eqref{48} that 

\begin{equation} \label{54}
\int_0^T \int_{B(x_0, R)} |\partial_tu_i(x,t)|^2 dx dt  \le   E_p(u_0) + 1.
\end{equation} 

\medskip
 
Thus if we set $T= R= R_k$, where $(R_k)_k$ is a sequence such that $R_k \to + \infty$, then by using the Cantor diagonal argument,  it follows from \eqref{53} and \eqref{54} that there exists a subsequence $(u_{i_k})_k$ and a map $u \in  C_{loc}^{1+ \beta, \hspace{0.5mm}\beta/p}(M\times[0,+\infty), \Omega)$  such that 
 
 \begin{equation*} 
 u_{i_k} \longrightarrow u \  \  \hbox{in} \  \  C^{1+ \beta', \hspace{0.4mm}\beta'/p}(B(x_0, R) \times[0,T), \Omega) \  \  \hbox{for all} \   R, T>0, \ 0 < \beta' < \beta,
 \end{equation*} 
 
 \noindent and 
   \begin{equation*} 
 \partial_tu_{i_k} \longrightarrow \partial_t u \  \  \hbox{weakly in} \  \  L^2(B(x_0, R) \times[0,T), \Omega) \  \  \hbox{for all} \   R, T>0. 
 \end{equation*} 
 
 It is easy to check that  by passing to the limit in \eqref{49} and using \eqref{46}  and \eqref{47}, $u$ is a solution of \eqref{15}. By passing to the limit in \eqref{50}, one obtains   formula  \eqref{17}  in Theorem \ref{th2}.

  \medskip
  
 When $N$ is compact, the convergence of the flow can be  proved in the same way as in the proof of Theorem \ref{th1}.  Indeed,  if we take $i=i_k$  in \eqref{54}   and pass to the limit when $k \to +\infty$, we obtain for any $R, T>0$
 
 \begin{equation*} 
\int_0^T \int_{B(x_0, R)} |\partial_tu(x,t)|^2 dx dt  \le   E_p(u_0) + 1.
\end{equation*} 

\noindent  which implies that by letting $T\to +\infty$ and $R \to +\infty$, 

  \begin{equation}\label{56}
 \int_0^{\infty}\int_M  |\partial_tu(x,t)|^2 dx dt  \le   E_p(u_0) + 1. 
 \end{equation}
 
 \noindent It follows from  \eqref{56} that there exits a sequence $t_j \to + \infty$ such that 
\begin{equation} \label{57}
\int_{M}|\partial_t u(x, t_j)|^2 dx \to 0   \  \  \hbox{as} \  \  t_j \to + \infty.
\end{equation}

On the other hand,   if we take $i=i_k$ in  \eqref{51}   and pass to the limit when $k \to +\infty$, we obtain  for any $R>0$
\begin{equation*}
\|\nabla u\|_{L^{\infty}(B(x_0, R) \times[0,+\infty))} \le C_R
\end{equation*} 

\medskip

\noindent which together with  the results of Dibenedetto \cite{Dibenedetto} imply, since $N$ is compact, 
\begin{equation} \label{58}
   \|u\|_{C^{1+ \beta,\hspace{0.3mm}\beta/p}(B(x_0, R)\times[0,+\infty))} \leq C_R.
 \end{equation} 

Hence  by taking a sequence $R_j\to \infty$, it follows from  \eqref{57}, \eqref{58}  and the Cantor Diagonal argument that the sequence $u(. , t_j)$  admits a subsequence  that converges in $C^{1+ \beta'}(M, \Omega)$  for all $\beta' < \beta$  to a map $u_{\infty}  \in C^{1+ \beta}(M, \Omega)$. By passing to the limit in equation \eqref{15} and using \eqref{57} we deduce that $u_{\infty}$ is a $p$-harmonic map satisfying $E_p(u_{\infty}) \le E_p(u_0)$. The proof of Theorem \ref{th2} is complete.

\end{proof}


\begin{proof}[ \textbf{\textit{Proof of Theorem \ref{th3} }}] 

Theorem \ref{th3} is a direct consequence of Theorem \ref{th2} since a manifold $N$ with nonpositive sectional curvature is a $\delta$-generalised regular ball for any $\delta >0$ (see Example \ref{exa1} ). In this case, it suffices to apply Theorem \ref{th2} by taking  any $\delta > \delta_p$. 

\end{proof} 

\bigskip


For the proof of Theorem \ref{th4} we need a modified version of Proposition \ref{p31} concerning solutions of the $p$-harmonic equation \eqref{14}. 

\bigskip

\begin{prop}\label{p42}
 Let $(M^m,g)$,  $(N^n,h)$ be two  Riemannian manifolds and let  $\Omega \subset N$  be a $\delta$-regular set.  Let  $u  \in C^{1}(M, \Omega)$ a $p$-harmoinc map and  set 
$$\varphi(x)=\frac{F(x,t)}{f^2(u(x)},$$ 
 where $F(x)= |\nabla u(x)|^2$ and $f$ is the function satisfying condition \eqref{c1}.  Then we have on the set $\big\{ x \in M \ : \ \nabla u(x) \not= 0 \big\}$, 
\begin{equation}\label{59}
\begin{split}
    & - \hbox{div}\left(F^{p-2 \over 2}   \nabla \varphi\right)  \leq  (p-2)\hbox{div}\left(F^{p-2 \over 2} B( d \varphi,.)\right)   - {1\over 40} F^{{p-4\over 2}}(f\circ u)^2|\nabla \varphi|^2  \\
     &- 2(\delta-\delta_p) {F^{{p\over 2}}\over (f\circ u)^4}\left|(\nabla f)\circ u \right|^2 |\nabla u|^2 + \  2 K_1 {F^{{p-2\over 2}}\over (f\circ u)^2} |\nabla u|^2, 
     \end{split}
\end{equation}
where  $\delta_p =   3(p-2)^2\Big( \sqrt{m} + 2p+ 2  \Big)^2 + 3$   \ and \  $-K_1 \le 0$  is a lower bound of the Ricci curvature of $M$. The tensor $B$ is defiened in Section \ref{s2} by \eqref{25} (by taking $\varepsilon = 0$). . 
\end{prop}

\medskip

\begin{proof} The proof is exactly the same as that of Proposition \ref{p31}. It is even  more easier since the parabolic term  $\partial_t u$ is not present. Nevertheless,  we have to consider only points $x \in M$ such that $\nabla u(x) \not=0$. The reason is that our $p$-harmonic map is  sufficienltly smooth at such points  to apply the elliptic version of Bochner formula \eqref{26}. 

\end{proof}

\begin{proof}[ \textbf{\textit{Proof of Theorem \ref{th4} }}] 

In this proof $C$ denotes a positive constant depending only on $M, \Omega$ and $p$, and whose value may change from line to line. Define the set $E$ by
$$E = \big\{ x \in M \ : \  \nabla u(x) \not=0 \big\}$$
which  is an open set of $M$ since $u \in C^{1}(M)$. By the regularity theory of elliptic equations  we have $u \in C^{\infty}(E)$.

\medskip

As in Proposition \ref{p42}, we set $\varphi=\frac{F}{f^2(u)}$,  where $F= |\nabla u|^2$, and $f$ is as in  \eqref{c1}.   For $\varepsilon >0$, let $\varphi_{\varepsilon} = (\varphi - \varepsilon)^{+} $. Then $\varphi _{\varepsilon}$ is a locally Lipschitz function on $M$  with support in $E$  and satisfies 

\begin{equation} \label{06}
\begin{cases} 
(i) \ \  \varphi_{\varepsilon}(x) = 0  \  \hspace{1,5cm} \hbox{if} \    \varphi(x) <  \varepsilon \cr
(ii)\ \   \varphi_{\varepsilon}(x)= \varphi(x) -\varepsilon \  \ \    \hbox{if} \   \varphi(x) \ge \varepsilon \cr 
(iii) \ \  \nabla  \varphi_{\varepsilon}(x) = \nabla \varphi(x) \  \       \hbox{if} \   \varphi(x) \ge \varepsilon \cr 
(iv) \  \  \nabla  \varphi_{\varepsilon}(x) =  0 \  \    \   \   \    \  \  \   \   \hbox{if} \   \varphi(x) <  \varepsilon. 
\end{cases}
\end{equation}

\medskip

Fix a point $x_0 \in E$ and  $R> 0$,  and   let $\phi_R \in C_0^{1}(B(x_0, 2R))$ such that 
\begin{equation}\label{60}
\begin{cases}
0 \le \phi_R \le 1 \cr 
\phi_R = 1  \  \  \hbox{on} \  \  B(x_0, R) \cr
|\nabla\phi_R | \le C R^{-1}.
\end{cases}
\end{equation}

\bigskip

We have  from Proposition \ref{p42}, since by hypothesis we have $K_1= 0$ ($M$ is supposed to be of nonnegative Ricci curvature) and $\delta > \delta_p$, 
\begin{equation} \label{61}
- \hbox{div}\left(F^{p-2 \over 2}   \nabla \varphi\right) + {1\over 40} F^{{p-4\over 2}}f^2(u)|\nabla \varphi|^2  \leq  (p-2)\hbox{div}\left(F^{p-2 \over 2} B(d \varphi,.)\right). 
\end{equation} 
   
   \medskip

 If we multiply \eqref{61} by $\phi_R^2 \varphi_{\varepsilon}\varphi^{-1} $ and integrate on $E$ by using \eqref{06} we have 
 
\begin{equation} \label{62}
    \begin{split}
   & \varepsilon\int_{E\cap B(x_0, 2R)} F^{p-2 \over 2} \phi_R^2 \varphi^{-2} |\nabla \varphi_{\varepsilon}|^2  dx  + {1\over 40}  \int_{E\cap B(x_0, 2R)} F^{{p-4\over 2}}f^2(u)\phi_R^2 \varphi_{\varepsilon} \varphi^{-1} |\nabla \varphi_{\varepsilon}|^2 \ dx \le  \\
   &  -2   \int_{E\cap B(x_0, 2R)} F^{p-2 \over 2} \varphi_{\varepsilon}  \phi_R \varphi^{-1} \nabla\varphi_{\varepsilon}\cdot\nabla \phi_R  dx  - \varepsilon(p-2) \int_{E\cap B(x_0, 2R)} F^{p-2 \over 2} \phi_R^2\varphi^{-2} B(d\varphi_{\varepsilon} ,d \varphi_{\varepsilon})   dx  \\
   & - 2 (p-2)   \int_{E\cap B(x_0, 2R)} F^{p-2 \over 2} \varphi_{\varepsilon}  \phi_R \varphi^{-1} B(d\varphi_{\varepsilon}, d\phi_R)  dx.
    \end{split}
\end{equation}

\medskip

 Since  by \eqref{250} we have $ B(d\varphi_{\varepsilon} , d\varphi_{\varepsilon}) \ge 0 $ and  $|B(d\varphi_{\varepsilon} , d\phi_R)|\le |\nabla\varphi_{\varepsilon}| |\nabla \phi_R|$, then it follows from \eqref{62} by using \eqref{60}  and the fact that 
 $F= {\varphi \over f^2(u)}$ : 
 
 \begin{equation*}
    \begin{split}
   & {39\over 40}\varepsilon\int_{E\cap B(x_0, 2R)} F^{p-2 \over 2} \phi_R^2 \varphi^{-2} |\nabla \varphi_{\varepsilon}|^2  dx  + {1\over 40}  \int_{E\cap B(x_0, 2R)} F^{{p-4\over 2}}f^2(u)\phi_R^2  |\nabla \varphi_{\varepsilon}|^2 \ dx    \le \\
   & C R^{-1} \int_{E\cap B(x_0, 2R)} F^{p-2 \over 2}\phi_R |\nabla \varphi_{\varepsilon}|  dx 
       \end{split}
\end{equation*}

\medskip
 
\noindent  which implies, since $C^{-1} \le f \le C$, 
 
  \begin{equation} \label{251}
 \int_{E\cap B(x_0, 2R)} F^{{p-4\over 2}}\phi_R^2  |\nabla \varphi_{\varepsilon}|^2 \ dx       \le C R^{-1} \int_{E\cap B(x_0, 2R)} F^{p-2 \over 2}\phi_R |\nabla \varphi_{\varepsilon}|  dx .
\end{equation}

 \medskip
 
 On the other hand, we have by the Cauchy-Schwarz inequality  
 
 $$  \int_{E\cap B(x_0, 2R)} F^{p-2 \over 2}\phi_R |\nabla \varphi_{\varepsilon}|  dx \le \left(\int_{E\cap B(x_0, 2R)} F^{p-4 \over 2}\phi_R^2 |\nabla \varphi_{\varepsilon}|^2  dx\right)^{1\over 2}\left(\int_{E\cap B(x_0, 2R)} F^{p \over 2}\ dx\right)^{1\over 2} .$$
 
\noindent Hence it follows from \eqref{251} that 

$$  \int_{E\cap B(x_0, 2R)} F^{{p-4\over 2}}\phi_R^2  |\nabla \varphi_{\varepsilon}|^2 \ dx  \le  CR^{-2} \int_{E\cap B(x_0, 2R)} F^{p \over 2}\ dx$$

\medskip

\noindent and  since $\phi_R = 1 $ on $B(x_0, R)$, then we obtain 
 
   \begin{equation} \label{64}
\int_{E\cap B(x_0, R)} F^{{p-4\over 2}} |\nabla \varphi_{\varepsilon}|^2 \ dx  \le  CR^{-2} \int_{E\cap B(x_0, 2R)} F^{p \over 2}\ dx \le CR^{-2}E_p(u).
\end{equation}

\medskip

Thus  by letting $R \to + \infty$ in \eqref{64} we obtain  $F^{p-2 \over 2} |\nabla \varphi_{\varepsilon}|^2 = 0$ on $E$, and then $ \nabla\varphi_{\varepsilon} = 0$ on $E$ since $ F >0$ on $E$. The constant $\varepsilon >0$ being arbitrary, we have then 

\begin{equation}  \label{65}
\nabla \varphi = 0 \ \hbox{on} \   E.
\end{equation}

\medskip

Our objective is to prove that $F= 0$ on $M$.   Suppose by contradiction that there exist $x_0 \in M$ such that $F(x_0) \not=0$, that is,  $x_0 \in E$. Let ${\mathcal C}_0 \subset E$  be the connected component of $x_0$ in $E$. Since ${\mathcal C}_0$ is open and connected, then  by \eqref{65}  $\varphi$  is constant on ${\mathcal C}_0$, that is, $F= \lambda_0 f(u)$ on ${\mathcal C}_0$  for some constant $\lambda_0 \ge 0$. But by \eqref{c1}, we have $f \ge C^{-1}$. Hence  we have 
\begin{equation} \label{66}
 F \ge \lambda_0 C^{-1} \  \  \hbox{on} \   {\mathcal C}_0.
 \end{equation}
 
\medskip

We distinguish two cases : $E =M$ and $E \not=M$.
 
\medskip

\underline{\bf Case 1 :  $E=M$.} Since $M$ is connected, then ${\mathcal C}_0= M$ in this case.   We have by hypothesis
$$\int_M F^{p\over 2} dx = \int_M |\nabla u|^p dx < + \infty$$
which implies by using \eqref{66}  that $\lambda_0 = 0$.  Hence $F= 0$ on $M$  contradicting the fact that $F(x_0) \not=0$. 

\medskip

\underline{\bf Case 2 :  $E\not=M$.}  In this case we have $\partial {\mathcal C}_0 \not= \emptyset$, where $\partial {\mathcal C}_0$ is the topological boundary of ${\mathcal C}_0$. It follows from \eqref{66}    by continuity of $F$  that   $F \ge \lambda_0 C^{-1}$ on $\partial {\mathcal C}_0$.   On the other hand, by the definition of a connected component  and since $E$ is open,  we have $\partial {\mathcal C}_0 \subset M\setminus E$.  This implies that $F=0$ on $\partial {\mathcal C}_0$, and then $\lambda_0 = 0$.  Thus we have $F=0$ on  ${\mathcal C}_0$ contradicting $F(x_0) \not= 0$. 

\bigskip

Therefore, we have  proved that $F =0$, which implies that $u$ is constant on $M$ since $M$ is connected.  The proof of Theorem \ref{th4} is complete.

\end{proof}

\bigskip


\end{document}